\documentclass[11pt,a4paper]{article}

\usepackage{amsmath, amsfonts, amssymb}
\usepackage{theorem}
\usepackage[dvips]{epsfig}
\usepackage{latexsym}
\usepackage{exscale}
\usepackage[latin1]{inputenc}
\newtheorem{thm}{Theorem}[section]
\newtheorem{lem}[thm]{Lemma}
\newtheorem{prop}[thm]{Proposition}

\newtheorem{cor}[thm]{Corollary}

\def\be#1\ee{\begin{equation}#1\end{equation}}
\newcommand{\bea}{\begin{eqnarray}}
\newcommand{\eea}{\end{eqnarray}}
\newcommand{\beaa}{\begin{eqnarray*}}
\newcommand{\eeaa}{\end{eqnarray*}}
\newcommand{\bei}{\begin{itemize}}
\newcommand{\eei}{\end{itemize}}
\newcommand{\bee}{\begin{enumerate}}
\newcommand{\eee}{\end{enumerate}}

\def\norm#1{\left\|#1\right\|}             
\def\abs#1{\left\vert #1 \right\vert}      

\def\set#1{\left\{#1\right\}}

\def\Z{{\mathbb Z}}
\def\N{{\mathbb N}}

\def\ex{\mathrm{e}}

\newcommand{\mA}{{\mathcal A}}

\def\al{\alpha}
\newcommand{\B}{{\mathcal{B}}}
\newcommand{\BB}{{\B_\infty}}

\newcommand{\dg}{\gamma}

\newcommand{\eps}{\varepsilon}

\newcommand{\II}{{\mathcal I}}

\def\L{{\mathcal{L}}_\mu}
\def\LL{\mathcal L}
\def\bbL{\mathbb L}

\def\on{{\mathbf 1}}

\newcommand{\mQ}{{\mathcal Q}}
\newcommand{\QQ}{(1-1/q)}

\def\s{\sigma}

\def\ro{\mathbf 0}

\newcommand{\thd}{{\B^\bullet_\mu}}   
\def\pq{\preceq}
\def\pc{\prec}
\def\qp{\succeq}

\def\PQ{\unlhd}
\def\PC{\lhd}

\newenvironment{proof}[1][] {\noindent {\bf Proof#1:} }{\hspace*{\fill}$\square$\medskip\par}

\begin{document}

\title{Compactness Properties
of  Weighted Summation Operators on Trees -- The Critical Case}
\author{Mikhail Lifshits \and Werner Linde
  }
 \date{\today}
    \maketitle
\begin{abstract}
The aim of this paper is to provide upper bounds for the entropy numbers of summation operators on trees in a critical case.
In a recent paper \cite{LifLin10} we elaborated a framework of weighted summation operators on general trees
where we related the entropy of the operator with those of the underlying tree equipped with an appropriate
metric. However, the results were left incomplete in a critical case of the entropy behavior,
because this case requires much more involved techniques. In the present article we fill the gap
left open in \cite{LifLin10}. To this end
we develop a method, working in the context of general trees and general weighted summation operators,
which was recently proposed in
\cite{Lif10} for a particular critical operator on the binary tree. Those problems  appeared in
natural way during the study of compactness properties of certain Volterra integral operators in a
critical case.
\end{abstract}
\bigskip
\bigskip
\bigskip
\bigskip
\bigskip

\vfill
\medskip


\baselineskip=6.0mm

\maketitle
\newpage

\section{Introduction}

Let $T$ be a tree with partial order structure $"\pq"$, i.e., one has $t\pq s$ whenever $t$
lies on the way leading from the root of $T$ to $s$. Suppose we are given two weight functions
$\al,\s :  T \mapsto (0,\infty)$ satisfying
\be
\label{bound}
\kappa:=\sup_{s\in T} \left(\sum_{v\pq s}\al(v)^q\right)^{1/q}\s(v)<\infty
\ee
for some $q\ge 1$. Then the weighted summation operator $V_{\al,\s}$ is well--defined by
\be
\label{Vals}
(V_{\al,\s}\mu)(t):= \al(t)\sum_{s\qp t}\s(s)\mu(s)\;,\quad t\in T\,,
\ee
for $\mu\in \ell_1(T)$ and it is bounded as operator from $\ell_1(T)$ to $\ell_q(T)$
with $\norm{V}\le \kappa$.

Our aim is to describe compactness properties of $V_{\al,\s}$. This turns out
to be a challenging problem because those properties of $V_{\al,\s}$ do not
only depend on $\al$, $\s$ and $q$ but also on the structure of the underlying tree. The
motivation for the investigation of those questions stems from \cite{Lif10} where compactness
properties of certain Volterra integral operators were studied; for the latter subject, see also \cite{Lin}.

The basic observation in \cite{LifLin10} is as follows: The weights $\al$ and $\s$ generate
in natural way a metric $d$ on $T$ and covering properties of $T$ by $d$--balls are tightly related
with the degree of compactness of $V_{\al,\s}$. To be more precise, for $t\pq s$ we define their
distance as
\be
\label{metric}
d(t,s):=\max_{t\pc v \pq s} \left(\sum_{t\pc \tau\pq v}\al(\tau)^q\right)^{1/q}\s(v)\;.
\ee
If $\s$ is non--increasing, i.e., $t\pq s$ implies $\s(t)\ge \s(s)$, then, as shown in
\cite{LifLin10}, the distance $d$ extends to a metric on the whole tree $T$.

Let $N(T,d,\eps)$ be the {\it covering numbers} of
$(T,d)$, i.e.,
$$
N(T,d,\eps):=\inf\set{n\ge 1 : T=\bigcup_{j=1}^n U_\eps(t_j)}
$$
with (open) $\eps$--balls $U_\eps(t_j)$ for certain $t_j\in T$
and let $e_n(V_{\al,\sigma})$ be the sequence of {\it dyadic entropy numbers} of $V_{\al,\s}$
defined as follows: If $X$ and $Y$ are Banach spaces and $V: X\mapsto Y$ is an operator,
then the $n$--th entropy number of $V$ is given by
$$
e_n(V):=\inf\left\{\eps>0 : \set{V(x) : \norm{x}_X\le 1}\;
           \mbox{is covered by at most}\;
           2^{n-1}\;\mbox{open}\; \eps\mbox{--balls in}\;Y\right\}\;.
$$
The operator $V$ is compact if and only if $e_n(V)\to 0$ as $n\to\infty$. Thus the
behavior of $e_n(V)$ as $n\to 0$ may be viewed as measure for the degree of
compactness of $V$.
We refer to \cite{CS} or \cite{ET} for more information about entropy numbers and
their properties.

One of the main results in \cite{LifLin10} asserts the following where, for
simplicity, we only formulate it for $1<q\le 2$ .\\
Suppose that
$$
N(T,d,\eps)\le c\,\eps^{-a}\abs{\log\eps}^b
$$
for some $a>0$ and $b \ge 0$. Then this implies
$$
e_n(V_{\al,\s} : \ell_1(T)\mapsto\ell_q(T))\le c'\,n^{-1/a-1/q}(\log n)^{b/a}\;.
$$
Hereby the order of the right--hand side is not improvable.

One may ask whether or not a similar relation between $N(T,d,\eps)$ and $e_n(V_{\al,\s})$
remains valid in the (probably more interesting) case that $N(T,d,\eps)$ tends  to infinity exponentially
as $\eps\to 0$. In
\cite{LifLin10} we were able to provide a partial answer to this question: some interesting critical case
remained unsolved. More precisely, we could prove the following (again we only formulate the result for $1<q\le 2$):\\
Assume
$$
\log N(T,d,\eps)\le c\,\eps^{-a}
$$
for some constant $c>0$ and some $a>0$. Then this implies
\be
\label{noncr1}
e_n(V_{\al,\s} : \ell_1(T)\mapsto\ell_q(T))\le c'\,n^{-1/q'}(\log n)^{1/q'-1/a}
\ee
provided that $a<q'$ while for $a>q'$ we have
\be
\label {noncr2}
e_n(V_{\al,\s} : \ell_1(T)\mapsto\ell_q(T))\le c'\,n^{-1/a}\;.
\ee
Here and in the sequel $q'$ denotes the conjugate of $q$ defined by $1/q'=1\,-1/q$. Again, both estimates
are of the best possible order.

The most interesting critical case $a=q'$ remained open. Here our methods only led to
\be
\label{extralog}
e_n(V_{\al,\s} : \ell_1(T)\mapsto\ell_q(T))\le c'\,n^{-1/a}(\log n)\;.
\ee
In view of the available lower estimates and suggested by the results in \cite{Lif10}, where a special
but representative case (binary trees, $q=2$ and
some special weights) was handled,  we conjectured that the logarithm on the right hand side of (\ref{extralog}) is
unnecessary.

The main aim of the present paper is to verify this conjecture. We shall prove the following.

\begin{thm}
\label{main}
Let $T$ be an arbitrary tree and let $\al,\s$ be weights on $T$ satisfying $(\ref{bound})$ for
some $q\in(1,2]$. Furthermore, assume $\s$ to be non--decreasing and let $d$ be the metric
on $T$ defined via $(\ref{metric})$.
Suppose
\be
\label{Nbound}
\log N(T,d,\eps)\le c_0\,\eps^{-q'}
\ee
with some $c_0>0$. Then it follows
$$
e_n(V_{\al,\s} : \ell_1(T)\mapsto \ell_q(T))\le c\,c_0^{1-\,1/q}\,n^{-\QQ}
$$
for some $c=c(q)$ independent of $\al$, $\s$ and $T$.
\end{thm}


As an illustration, let us show how the main result of \cite{Lif10} appears to be
a direct corollary of Theorem \ref{main}.

\begin{cor}
Let $T$ be an infinite binary tree and let the weights on $T$ be defined by
\[
    \s(t)=1,\  \al(t)=(|t|+1)^{-1},\qquad t\in T,
\]
where $|\cdot|$ denotes the order of an element in the tree (cf.~$(\ref{deforder})$ below).
 Then there exists a finite positive $c$ such that for all positive integers $n$
 we have
$$
   e_n(V_{\al,\s} : \ell_1(T)\mapsto \ell_2(T))\le c \, n^{-1/2}.
$$
\end{cor}

\begin{proof}
Take any positive integer $k$ and let $T_k$ be the union of the levels less or equal than $k$,
i.e., $T_k:=\set{t\in T : |t|\le k}$.
Then
\[
  \# T_k = \sum_{j=0}^k 2^j < 2^{k+1}.
\]
On the other hand, it is obvious that (with $q=2$ in the definition of the metric $d$)
\[
   \sup_{s\in T} \inf_{t\in T_k} d(t,s)^2 < \sum_{j=k+1}^\infty (j+1)^{-2} \le (k+1)^{-1}.
\]
It follows that $\log N(T,d,(k+1)^{-1/2})\le (k+1)\cdot \log 2$, thus Theorem \ref{main} applies
with $q=2$ and yields the assertion of the corollary.
\end{proof}
\medskip

Let us shortly explain why the methods of \cite{LifLin10} are \textit{not} appropriate for
the proof of Theorem \ref{main} and why, therefore,  a completely new approach is needed.
A basic step in the proof of (\ref{noncr1}) and (\ref{noncr2}) is an estimate
for the entropy of the convex hull of a certain subset of $\ell_q(T)$. Estimates of
this type are well--known (see e.g.~\cite{CKP}
or \cite{CrSt}). But here a critical case appears which exactly corresponds to $a=q'$ in our
situation. As shown  for $q=2$ in \cite{Gao} and  for type $q$--spaces in \cite{CrSt}, in
this critical case estimates of
the entropy of convex hulls give an extra
$\log$--term which, in general, cannot be avoided. Thus, in order to prove Theorem \ref{main}
one has to show that such an extra $\log$--term does not appear for the entropy of convex hulls
provided the investigated sets are related to weighted summation operators on trees. This demands
a completely new approach which was for the first time used in \cite{Lif10} and which we elaborate
here further. The basic idea is to approximate an operator  defined on a Banach space $X$ (in
our situation we have $X=\ell_1(T)$ ) by a family
of operators depending on the elements in $X$. To this end one has to control at the same time
the entropy numbers
of the approximating operators as well as the number of those operators.

\section{Trees}
\setcounter{equation}{0}

Let us recall some basic notations related to trees which will be used later on.
In the sequel $T$ always denotes a finite or an infinite
tree. We suppose that $T$ has a unique root which we denote by
$\ro$ and that each element $t\in T$ has a finite number
of offsprings. Thereby we do not exclude that some elements do not possess any offspring, i.e.,
the progeny of some elements may "die out". The tree structure leads in natural way
to a partial order $,\!,\pq "$ by letting $t\pq s$, respectively $s\qp t$,  provided there are
$t=t_0, t_1,\ldots, t_m=s$ in $T$ such that for $1\le j\le m$ the
element $t_j$ is an offspring of $t_{j-1}$.
The strict inequalities have the same meaning with the additional assumption $t\not=s$.
Two elements $t,s\in T$ are said to be {\it comparable} provided that either $t\pq s$ or $s\pq t$.

Given $t\in T$ with $t\not=\ro$ we denote by $t^-$ the (unique) parent element of $t$,
i.e., $t$ is supposed to be an offspring of $t^-$.

For $t,s\in T$ with $t\pq s$ the order interval $[t,s]$ is defined by
$$
[t,s]:= \set{ v\in T : t\pq v\pq s}
$$
and in a similar way we construct $(t,s]$ or $(t,s)$ .

A subset $B\subseteq T$ is said to be a {\it branch} provided that all elements in $B$ are
comparable and, moreover, if $t\pq v\pq s$ with $t,s\in B$, then this implies $v\in B$
as well. Of course, finite branches are of the form $[t,s]$ for suitable $t\pq s$.

A set $B\subseteq T$ is called a {\it tree} provided it is a tree w.r.t.~the structure of $T$;
in particular,
if $r\in B$ is its root, then for each $s\in B$ with $s\qp r$ it holds $[r,s]\subseteq B$.
If $\ro$ is the root of the tree, then $B$ is called a \textit{subtree} of $T$. Given a tree $B\subseteq T$
an element $t\in B$ is said to be {\it terminal} provided that $s\notin B$ for all offsprings
$s$ of $T$.

Finally, for any $s\in T$ its {\it order} $|s|\ge 0$ is defined by
\be
\label{deforder}
|s|:=\#\set{t\in T :  t \pc s}\;.
\ee

\section{Reduction of the Problem}
\setcounter{equation}{0}

An easy scaling argument shows that we may assume that estimate (\ref{Nbound}) holds with $c_0=1$.
Another quantity that naturally appears in our bounds is $\kappa$ defined in (\ref{bound}). Notice
that for any $\eps>\kappa$ we have $N(T,d,\eps)\ge 2$. Therefore, (\ref{Nbound}) yields
$\kappa \le \left(\frac{c_0}{\log 2}\right)^{1-\,1/q}$, hence the scaling of $c_0$ implies that
$\kappa>0$ is uniformly bounded. After this scaling is done, all the constants appearing in
the proof of Theorem \ref{main} only depend on $q$. We will denote them
by $c$ without further distinction.
\medskip

\noindent
\textit{First reduction}:
A first important simplification of the problem is as follows:\\
Without losing generality we may assume that $\s$ attains only values in
$\set{2^{-k} : k\in \Z}$. Although this has been proved in \cite{LifLin10}, let
us shortly repeat the argument. Set
$$
I_k:=\set{t\in T : 2^{-k-1}<\s(t)\le 2^{-k}}
$$
and define a new weight $\hat \s$ by
$$
\hat \s:= \sum_{k\in\Z} 2^{-k}\on_{I_k}\;.
$$
Then $e_n(V_{\al,\hat \s})\le e_n(V_{\al,\s})$ and if $\hat d$ denotes the metric on $T$
defined via (\ref{metric})  by $\al$ and $\hat \s$ instead of by $\al$ and $\s$,
then $N(T,\hat d,\eps)\le N(T,d,2\eps)$.
Moreover, if $\s$ is non-increasing,  then $\hat \s$ also has this property.
Clearly, this shows that it suffices to prove Theorem \ref{main} for weights $\s$ only
attaining values in $\set{2^{-k} : k\in \Z}$.
\medskip

\noindent
\textit{Second reduction}: Suppose that $\s$ is of the special form
\be
\label{red}
\s=\sum_{k\in \Z} 2^{-k}\on_{I_k}
\ee
for certain disjoint $I_k\subseteq T$, $k\in\Z$. Since $\s$ is assumed to be non--decreasing, the
collection
$\II:=(I_k)_{k\in\Z}$ of subsets in $T$ possesses the following properties:
\bee
\item
The $I_k$ are disjoint and $T=\bigcup_{k\in\Z} I_k$, i.e., $\II$ is a partition of $T$ .
\item
For each $s\in T$ the set $I_k\cap [\ro,s]$ is either empty or a branch. Moreover, the
$I_k$ are ordered in the right way, i.e., if $l<k$
and $v_l\in I_l\cap [\ro,s]$ and $v_k\in I_k\cap[\ro,s]$, then this implies $v_l\pc v_k$.
\item
Since $\s$ is bounded, of course, it follows that $I_k=\emptyset$ whenever $k\le k_0$ for a certain $k_0\in\Z$.
\eee

Using this partition $\II$  of $T$ we construct an operator $W:\ell_1(T)\mapsto \ell_q(T)$
which may be viewed as a localization of $V_{\al,\s}$.  It is defined as follows:
If $\mu\in\ell_1(T)$, then
\be
\label{defW}
(W\mu) (t):= \al(t)\sum_{s\qp t\atop{s\in I_k}}\s(s)\mu(s)= \al(t)\,2^{-k}\sum_{s\qp t \atop{s\in I_k}} \mu(s),
\qquad t\in I_k\ .
\ee
Note that for each $k\in\Z$ and $t\in I_k$ the value of $(W\mu)(t)$ depends only on the values of
$\mu(s)$ for $s\qp t$ and $s\in I_k$. This is in complete contrast to the definition of $V_{\al,\s}$ because here
the value of $(V_{\al,\s}\mu)(t)$ depends on the values of $\mu(s)$ for \textit{all} $s\qp t$. Nevertheless, as shown in
\cite[Proposition 4.3]{LifLin10} the following is valid.

\begin{prop}
\label{red2}
It holds
$$
e_n(V_{\al,\s}:\ell_1(T)\mapsto\ell_q(T))\le 2\,e_n(W : \ell_1(T)\mapsto \ell_q(T))\;.
$$
\end{prop}

As a consequence we see that it suffices to estimate the entropy numbers of $W$ suitably.
\medskip

\noindent
\textit{Third reduction}: Given $\eps>0$ a set $S\subset T$ is said to be an $\eps$--order net
provided that for each $t\in T$ there is an $s\in S$ with $s\pq t$ and $d(s,t)<\eps$. Let
$$
\tilde N(T,d,\eps):=\inf\set{\# S : S\;\mbox{is
an}\:\eps\mbox{--order net of}\;T}
$$
be the corresponding order covering numbers. Clearly, we have
$$
N(T,d,\eps)\le \tilde N(T,d,\eps)\;.
$$
But, surprisingly, also a reverse estimate holds as shown in \cite[Proposition 3.3]{LifLin10}.
More precisely, we always have
$$
\tilde N(T,d,2\eps)\le N(T,d,\eps)\;.
$$
\medskip

Summing up it follows that it suffices to prove the following variant of Theorem \ref{main}:

\begin{thm}
\label{main1}
Suppose that $\s$ is non--increasing and of the special form $(\ref{red})$. Define $W$ as in $(\ref{defW})$
with respect to the partition $\II$ generated by $\s$. If
$$
\log \tilde N(T,d,\eps)\le \eps^{-q'}
$$
for some $q\in (1,2]$, then this implies
$$
e_n(W: \ell_1(T)\mapsto \ell_q(T))\le c\,n^{-\QQ}\;.
$$
\end{thm}

\section{Proof of Theorem \ref{main1}}
\setcounter{equation}{0}

\label{proof}

We start with explaining the strategy for proving Theorem \ref{main1}. This is done
by a quite general approximation procedure which, to our knowledge,  for
the first time appeared in \cite{Lif10}.

\begin{prop}
\label{apprV}
Let $V$ be a bounded linear operator between the Banach spaces $X$ and $Y$ and let
$\{V_\gamma : \gamma\in\Gamma\}$ be a (finite) collection of operators
from $X$ to $Y$. Setting $M:=[\log_2(\# \Gamma)]+1$, for each $k\ge 1$ it follows that
\be
\label{appr}
e_{k+M}(V)\le \sup_{\gamma\in\Gamma} e_k(V_\gamma)
\ + \
\sup_{\norm{x}_X\le 1}\inf_{\gamma\in \Gamma}
\norm{V x - V_\gamma x}_Y\;.
\ee
\end{prop}

How this proposition is applied ? Let $a>0$ and suppose that for each $n\ge 1$ there exist
operators $\{V^n_\gamma : \gamma \in \Gamma_n\}$ from $X$ to $Y$
such that $\log(\# \Gamma_n)\le c_1\,n$ and $e_{[\rho n]}(V^n_\gamma)\le c_2\, n^{-a}$ for some $\rho\ge 1$. If,
furthermore, for each $x\in X$ with $\norm{x}_X\le 1$ there is a $\gamma=\gamma(x)\in\Gamma_n$ with
$$
\norm{V x -V^n_\gamma x}_Y\le c_3\,n^{-a}\;,
$$
then an application of Proposition \ref{apprV} with $k=[\rho n]$ leads to
$e_{c_4\,n}(V)\le c_5\,n^{-a}$ for $n\in\N$, hence by the monotonicity of entropy numbers to
$e_n(V)\le c\,n^{-a}$ for $n\ge 1$.
\medskip

Thus, in order to apply this general approximation scheme to $W$   defined in (\ref{defW}) and with $a=1-\,1/q$,
for each $n\ge 1$ we have to construct a suitable collection $\{W_\gamma^n : \gamma\in\Gamma_n\}$ of operators from $\ell_1(T)$
to $\ell_q(T)$ with $\log(\# \Gamma_n)\le c_1\,n$,
\be
\label{Wgamma}
\inf_{\gamma\in \Gamma_n}\norm{W\mu -W^n_\gamma \mu}_q\le c_3\,n^{-\QQ}\,,\quad \norm{\mu}_1\le 1\;,
\ee
such that
\be
\label{enrho}
e_{[\rho n]}(W^n_\gamma)\le c_2\,n^{-\QQ}\,,\quad n\ge 1\,,
\ee
 for a certain $\rho\ge 1$.

Let us shortly describe the strategy of this quite involved construction procedure.
In a first step, we build an auxiliary structure on the tree $T$. Namely, we construct a system
$(\B_m)_{m\ge 0}$ of refining tree partitions of $T$ based on the weights $\al$ and $\s$. This is done
in Subsections \ref{sub41} and \ref{sub42}.

Next, given $n\in\N$, we construct a set $\bbL_n$ of partitions of $T$
that will play the role of the parametric set $\Gamma_n$ mentioned above.
Namely,  for any $\mu\in \ell_1(T)$
with $\norm{\mu}_1\le 1$ we construct a special partition $\LL_\mu=\LL_\mu(n)$ of $T$. Any element
of the partition $\LL_\mu$ belongs to a suitable partition $\B_m$ built in the first step.
This construction is exposed in Subsection \ref{sub43}.
We let $\bbL_n=\{\LL_\mu\, : \, \mu\in \ell_1(T),\  \norm{\mu}_1\le 1\}$.
The size of  $\bbL_n$ remains under exponential control as required above.

Furthermore, each partition $\LL\in \bbL_n$ generates a representation
$W=\sum_{i=1}^4 W_\LL^i$, as explained in (\ref{Wasasum}) below.
We show that the sum $\sum_{i=1}^3 W_\LL^i$ can be used as approximating operators
as formulated in (\ref{Wgamma}) and which
admit the appropriate bound for the entropy numbers as stated in (\ref{enrho}).
Algebraic properties of the entropy numbers imply that it suffices to verify
$e_n(W_\LL^i)\le c_i\,n^{-\QQ}$ for $i=1,2,3$. The proof of these estimates
will be presented in Subsection \ref{sub44}. Surprisingly, thereby each of the three operators
must be treated by a different method.
\medskip

Hence, let us start with the investigation of
a special type of partitions of $T$.

\subsection{Tree Partitions}
\label{sub41}

Suppose we are given a subset $R\subseteq T$ with $\ro\in R$. If $r\in R$, set
$$
B_r :=\set{ s\in T : s\qp r\;\mbox{and}\; (r,s]\cap R=\emptyset}\;.
$$
Then $B_r$ is a tree in $T$ with root $r\in R$. Letting $\B:=\set{B_r : r\in R}$, the family
$\B$ is a partition of $T$ where each partition element is a tree. We call  $\B$ a
{\it tree partition} of $T$. Notice that each tree partition of $T$ may be represented in the
way described before with $R$ being the set of roots of $B\in\B$.

Given two tree partitions $\B_1$ and $\B_2$, we say that $\B_2$ refines $\B_1$
provided that each $B_2\in\B_2$ is contained in a suitable $B_1\in\B_1$.
Clearly, this is equivalent to $R_1\subseteq R_2$ with
generating (or root) sets $R_1$ and $R_2$ of $\B_1$ and $\B_2$, respectively.

Suppose now that $(\B_m)_{m\ge 0}$ is a sequence of tree partitions satisfying
$\B_0=\set{T}$ such that for each $m\ge 1$ the partition $\B_m$ refines $\B_{m-1}$.
Of course, this is equivalent to
$$
\{\ro\}=R_0\subseteq R_1\subseteq \cdots
$$
with the $R_m$ being the corresponding root sets. In order to distinguish the sets in
different levels let us write
$$
\B_m=\set{B_{r,m} : r\in R_m}
$$
where $B_{r,m}$ is an element of $\B_m$ with the root $r\in R_m$. Finally, set
\be
\label{defBinf}
\BB := \set{ B_{r,m} : r\in R_m\,,\; m\ge 0}\;.
\ee
Given $B_{r,m}$ and $B_{r',m'}$ in $\BB$ we say that the latter tree
is an offspring of the first one provided that $m'=m+1$ and, moreover,
$B_{r',m'}\subseteq B_{r,m}$. In that way $\BB$ becomes a tree with
root $B_{\ro,0}=\{T\}$. If we denote the generated partial order in $\BB$
by $"\PQ"$ resp.~by $"\PC"$ for the strict order, then $B_{r,m}\PQ B_{r',m'}$ if and only if $m'\ge m$ and
$B_{r',m'}\subseteq B_{r,m}$. Notice a minor abuse of notation:
we may have $B_{r,m}=B_{r,m'}$ as sets but they are equal in $\BB$ only if $m=m'$, i.e., the same
set may appear in different levels and is then treated as a multitude of different elements
in $\BB$.

In particular, all notions concerning trees apply to $\BB$. For example,
if we define the order of an element in $\BB$ as in (\ref{deforder}), then we
have
$$
\set{B\in \BB : \abs{B}=m}=\B_m\;.
$$
Let us still mention a special property of $\BB$. Suppose that $m'\ge m$. Then,
if $B,B'\in\BB$ have order $m$ or $m'$, respectively, then either $B\PQ B'$ or
$B\cap B'=\emptyset$.

\subsection{Construction of Tree Partitions}
\label{sub42}

Suppose we are given weights $\al$ and $\s$ on $T$ where $\s$ is assumed to
be as in (\ref{red}) with partition $\II =(I_k)_{k\in\Z}$. Let $d$ be the metric
constructed by (\ref{metric}) w.r.t.~$\al$, $\s$ and $q$. One of the main difficulties
is that the metric $d$ and the partition $\II$ do not match. For example, if $t\pq s$,
$t\in I_{k-1}$ and $s\in I_k$, then we get
\be
\label{example}
d(t,s)= \max\Big\{ 2^{-(k-1)}\Big(\sum_{t\pc v\pc \lambda(s)}\al(v)^q\Big)^{1/q},
2^{-k}\Big(\sum_{t\pc v\pq s}\al(v)^q\Big)^{1/q}\Big\}
\ee
where $\lambda(s)$ is defined by $I_k\cap [\ro,s]=[\lambda(s),s]$. Since we do not have
any information about the inner sums, this expression is difficult to handle.
Observe that for general $t,s\in T$ with $t\pq s$ expression (\ref{example}) becomes even more complicated.
Therefore we modify $d$ in a way better suited to $\II$
and set
$$
d_\II(t,s):= \left\{
\begin{array}{ccc}
\min\set{d(\lambda(s),s),d(t,s)}&:& t\pq s\\
+\infty&:&\mbox{otherwise}
\end{array}
\right.
$$
with $\lambda(s)$ as before. Let us reformulate this expression.
To this end we define an equivalence relation on $T$ by setting
$t\equiv s$ provided that there is a $k\in\Z$ with $t,s\in I_k$. Then,
if $t\pq s$, we may write $d_\II(t,s)$ as
$$
d_\II(t,s)=\s(s)\left(\sum_{t\pc v\pq s\atop{v\equiv s}}\al(v)^q\right)^{1/q}\;.
$$
Thus $d_\II$ may be viewed as a localization of $d$. Although
in general $d_\II$ is not a metric on $T$, we will use it later on to measure some
"distances".
\medskip

We suppose now that $T$ is a tree satisfying
\be
\label{logN}
\log \tilde N(T,d,\eps)\le \eps^{-q'}\;.
\ee
The next objective is to construct tree partitions suited to our problem. We do so by defining the
corresponding root sets.
For each $m\ge 1$ set
$$
\eps_m := (m\,\log 2 )^{-\QQ}\;.
$$
In view of (\ref{logN}) this choice immediately provides
\be \label{tNeps_m}
 \tilde N(T,d,\eps_m)\le 2^m.
\ee

\begin{prop}
\label{construct}
Suppose $(\ref{logN})$.
Then there are subsets $(R_m)_{m\ge 0}$ of $T$ possessing the following properties:
\bee
\item
It holds
$$
\set{\ro}=R_0\subseteq R_1\subseteq\cdots\;.
$$
\item
For each $m\ge 0$ one has \be \label{sizeRm}
\# R_m \le 2^{m+1}\;.
\ee
\item
The sets $R_m$ are $\eps_m$--order nets with respect to $d_\II$, i.e.
\be
\label{net}
\sup_{s\in T}\min_{r\in R_m} d_\II(r,s)<\eps_m\;.
\ee
\item
The sets $R_m$ are minimal w.r.t.~the order in $T$, i.e., whenever $\tau\in R_m\setminus R_{m-1}$,
hence $\tau\not=\ro$, thus $\tau^-$ is well--defined,
then $R_m^\tau := (R_m\setminus\set{\tau})\cup\set{\tau^-}$ does no longer satisfy $(\ref{net})$.
\eee
\end{prop}

\begin{proof}
We construct the sets $R_m$ by induction. Take $R_0:=\set{\ro}$ and suppose that for some $m\ge 1$ we already
have defined $R_0,\ldots,R_{m-1}$ possessing the desired properties.
Let
$$
   \mA_m:=\set{A\subseteq T : \# A \le 2^m\;\mbox{and}\; \sup_{s\in T}\min_{r\in A\cup R_{m-1}} d_\II(r,s)<\eps_m}\;.
$$
First of all, we establish that $\mA_m\not=\emptyset$.
Indeed, due to (\ref{tNeps_m}) there exist order $\eps_m$--nets of cardinality less or equal than $2^m$. Moreover, any such
net belongs to $\mA_m$, due to the definition of order nets and to the inequality  $d_\II(r,s)\le d(r,s)$ that
holds whenever $r\pq s$.

Next define a subclass $\mA^0_m$ of $\mA_m$ by
$$
   \mA^0_m:=\set{A\in\mA_m : \# A \;\mbox{is minimal}}
$$
and distinguish between the two following cases:\\
\textit{First case:}\ $\emptyset\in\mA^0_m$. This happens whenever $R_{m-1}$ satisfies
(\ref{net}) not only for $\eps_{m-1}$ but also for $\eps_m$. In that case we set $R_m:=R_{m-1}$.
Of course, the first three properties are satisfied and the fourth one holds by trivial reason.

\noindent
\textit{Second case:}\ $\emptyset\notin \mA^0_m$. Then all sets in $\mA^0_m$ have the same
positive cardinality  $p\le 2^m$.
For any  $A \in \mA_m^0$ let $F(A):= \sum_{\tau\in A}|\tau|$ and choose a set $A_m^*\in \mA_m^0$
such that
\[
 F(A_m^*)=\min_{A\in \mA_m^0} F(A)\;.
\]
Set
$$
   R_m := R_{m-1}\cup A_m^*\;.
$$
Clearly, then $R_{m-1}\subseteq R_m$ holds true.
Next,
$$
  \# R_m =\# R_{m-1} +\# A_m^* \le 2^m +p\le  2^m+2^m=2^{m+1},
$$
as asserted in property (2). Since $A_m^*\in \mA_m$, condition (\ref{net}) required in property (3) holds as well.

It remains to check property (4).  Fix any $\tau\in A_m^*$.
Note that necessarily $\tau \notin R_{m-1}$
because otherwise, by dropping $\tau$, we could diminish the cardinality of the set
and stay in $\mA_m$. In particular, we get $\tau\not=\ro$.

Next, consider the set  $A_m^\tau:=(A^*_m\setminus \set{\tau})\cup\set{\tau^-}$.
Clearly, $F(A_m^\tau)<F(A_m^*)$. Since $F$ attains its minimum on $\mA_m^0$ at $A_m^*$,
we infer that $A_m^\tau\not \in \mA_m^0$. Moreover, since $\# A_m^\tau =\# A^*_m =p$,
it follows that  $A_m^\tau\not \in \mA_m$. Consequently, because of $A_m^*=R_m\setminus R_{m-1}$,
property (4) of the proposition
holds.
This completes the proof.
\end{proof}
\medskip

Before proceeding further,
let us recall that the $R_m$ constructed above lead to tree partitions $\B_m$ of $T$
with $\B_m=\set{B_{r,m} : r\in R_m}$ where $s\in B_{r,m}$ if and only if $s\qp r$ and
$(r,s]\cap R_m=\emptyset$. We have
\be
\label{net2}
    d_\II(r,s)<\eps_m\quad \textrm{whenever} \ s\in B_{r,m}\, .
\ee
Moreover, each partition $\B_m$ refines $\B_{m-1}$.
\medskip

Next we turn to the crucial estimate for the tree partitions constructed in
Proposition \ref{construct}. Here property (4) plays an important role.

\begin{prop}
\label{crucial}
Let the $R_m$ be as in Proposition $\ref{construct}$ with corresponding tree partitions $\B_m=\set{B_{r,m} :r\in R_m}$.
Fix $m\ge 1$ and let $r\in R_{m-1}$, $\tau\in R_m$ with $r\pc \tau$ and $\tau\in B_{r,m-1}$. Then this implies
$$
\sigma(\tau)\left(\sum_{r\pc v\pq \tau^-\atop{v\equiv \tau}}\al(v)^q\right)^{1/q}\le c\,m^{-1}\;.
$$
\end{prop}
\begin{proof}
First note that $\tau\in B_{r,m-1}$ yields $\tau\notin R_{m-1}$, consequently $\tau\in R_m\setminus R_{m-1}$
and we are in the situation of property (4) of Proposition \ref{construct}. Hence there is an
$s\in T$ such that with $R_m^\tau := (R_m\setminus\set{\tau})\cup\set{\tau^-}$ we get
\be
\label{Rtau}
\min_{v\in R_m^\tau} d_\II(v,s)\ge \eps_m\;.
\ee
On the other hand, by property (3) of Proposition \ref{construct} it holds
\be
\label{Rm}
\min_{v\in R_m} d_\II(v,s)<\eps_m\;.
\ee
Since $R_m$ and $R_m^\tau$ differ only by one point, namely, by $\tau$ or $\tau^-$, respectively,
the two minima in (\ref{Rtau}) and (\ref{Rm}) may only be different if
$\tau$ is the maximal element in $[\ro,s]\cap R_m$. This implies  $s\in B_{\tau,m}$ as well as
\be
\label{dItau}
d_\II(\tau^-,s)\ge \eps_m\quad\mbox{and}\quad d_\II(\tau,s)<\eps_m\;.
\ee
Furthermore, it follows that $\tau\equiv  s$. Indeed, if they would not be
equivalent, then this would imply
$$
d_\II(\tau,s) =d(\lambda(s), s)= d_\II(\tau^-,s)
$$
which contradicts (\ref{dItau}). Recall that for $s\in I_k$  we denoted by $\lambda(s)$ the minimal element
in $[\ro,s]\cap I_k$.
Now from $s\equiv \tau$ we derive
\[
  \eps_m^q \le d_\II(\tau^-,s)^q= \sigma(s)^q\sum_{\tau^- \prec v\preceq s}\alpha(v)^q.
\]
On the other hand, we have $s\in  B_{\tau,m}\subseteq B_{r,m-1}$, hence by (\ref{net2})
\[
\sigma(s)^q\sum_{{ r \prec v\preceq s\atop v\equiv s}}\alpha(v)^q =d_\II(r,s)^q< \eps_{m-1}^q.
\]
By substraction we obtain
\begin{eqnarray*}
\sigma(\tau)^q  \sum_{{r \prec v\preceq \tau^- \atop v\equiv \tau}}\alpha(v)^q
  &=& \sigma(s)^q  \sum_{{r \prec v\preceq \tau^- \atop v\equiv s}}\alpha(v)^q
\\
&=& \sigma(s)^q
  \left(  \sum_{{r \prec v\preceq s \atop v\equiv s}}\alpha(v)^q
     -  \sum_{\tau^- \prec v\preceq s}\alpha(v)^q   \right)
\\
&\le&  \eps_{m-1}^q-\eps_m^q = c\left((m-1)^{-(q-1)}- m^{-(q-1)} \right)\le c \ m^{-q},
\end{eqnarray*}
as required.
\end{proof}

\subsection{Heavy and Light Domains}
\label{sub43}
Suppose we are given a sequence $\B_m$ of tree partitions as before, i.e., $\B_0=\{T\}$
and $\B_m$ refines $\B_{m-1}$. In the moment those tree partitions may be quite general,
but later on we will take the special partitions constructed via Proposition \ref{construct}.

We fix a number $n\ge 1$ and everything done now will depend on this number $n$
(although we do not reflect this dependence in the notation).
Let
$\mu\in\ell_1(T)$ be an arbitrary non--zero element with $\norm{\mu}_1\le 1$.
Define $\BB$ as in (\ref{defBinf}).
A subset $B\in\BB$ is said to be \textit{heavy} (w.r.t.~$\mu$) provided that
$$
\abs{\mu}(B)>\frac{\abs{B}}{n}
$$
where here and later on $\abs{\mu}(B):=\sum_{t\in B}\abs{\mu(t)}$.
Recall that $\abs{B}=m$ means that $B\in\B_m$. Otherwise, i.e., if
\be
\label{light}
\abs{\mu}(B)\le\frac{\abs{B}}{n}\;,
\ee
we call $B\in\BB$ \textit{light}. If
\be
\label{heavyall}
\thd:=\set{ B\in \BB : B\;\mbox{is heavy w.r.t.}\;\mu}\;,
\ee
it follows that $\thd\subseteq\set{B\in \BB : \abs{B}\le n}$. We have $B_{\ro,0}\in\thd$ and,
moreover, whenever $B\in\thd$ and $B'\in\BB$ satisfy $B'\PQ B$, then this implies $B'\in\thd$ as well.
In different words, $\thd$ is a subtree of $\BB$ and we call it, in accordance with the terminology of \cite{Lif10},
the {\it essential tree} in $\BB$ (w.r.t.~$\mu$ and $n$).
\bigskip

Among the light subsets of $\BB$ we choose the extremal ones as follows:
\be
\label{extlight}
\L:=\set{L\in\BB : L \;\mbox{is light and all }\;B\in\BB\;\mbox{with}\; B\PC L\;\mbox{are heavy}}\;.
\ee
In different words, a set $B_{r,m}\in\BB$ belongs to $\L$ if and only if $B_{r,m}$ is light and
each $B_{r',m'}\in\BB$ with $m'<m$ and $B_{r,m}\subseteq B_{r',m'}$ is heavy. Of course, it
suffices if this property is valid for $m'=m-1$.

\begin{prop}
\label{p1}
The set $\L\subseteq\BB$ is a tree partition of $T$.
\end{prop}
\begin{proof}
First we show that for $L,L'\in\L$ we either have $L=L'$ or $L\cap L'=\emptyset$.
Thus let us suppose $L\not=L'$. If $L\in\B_m$ and $L'\in\B_{m'}$, then $m=m'$ yields
$L\cap L'=\emptyset$ and we are done. Assume now $m'<m$. In that case either
$L\cap L'=\emptyset$ or $L'\PC L$. But since $L\in\L$ and $L'$ is light, the
latter case is impossible and this shows $L\cap L'=\emptyset$ as asserted.

Take now an arbitrary $s\in T$ and let $B_m(s)\in\B_m$ be the unique set in $\B_m$
with $s\in B_m(s)$. Then
$$
T=B_0(s)\PQ B_1(s)\PQ\cdots\;.
$$
Since $B_0(s)$ is heavy, there is a smallest $m_0=m_0(s)$ such that $B_{m_0}(s)$ is
light. Of course, $B_{m_0}(s)\in\L$ showing that
$T=\bigcup_{L\in \L} L$. This completes the proof.
\end{proof}

\begin{remark}
\rm
Observe that $\L$ is completely determined by $\thd$. Indeed, we have
$L\in \L$ if and only if $L\notin \thd$ and there is a $B\in\thd$
such that $L$ is an offspring of $B$. In different words, given
$\mu_1,\mu_2\in \ell_1(T)$, then it follows that $\LL_{\mu_1}=\LL_{\mu_2}$ if and only if
$\B^\bullet_{\mu_1}=\B^\bullet_{\mu_2}$.
\end{remark}
\medskip

The size of each essential tree and the total number of all possible essential
trees (if we let $\mu$ vary) are strongly bounded. For completeness, let us repeat the corresponding arguments
from \cite{Lif10}.

\begin{lem} \label{lb1}
Let $\mu\in\ell_1(T)$ with $\|\mu\|_1\le 1$ and denote by
${\mQ}\subseteq\BB$ the set of terminal domains of the subtree $\thd$.
It is true that
\begin{equation}\label{bsizeTmuQ}
\sum_{B\in {\mQ}} |B| < n.
\end{equation}
Moreover,
\begin{equation}\label{bsizeTmu}
\# \thd  \le n.
\end{equation}
\end{lem}

\begin{proof} Since all terminal domains are disjoint and they are all heavy,
we have
\[
1 \ge \norm{\mu}_1 \ge \sum_{B\in {\mQ}} |\mu|(B) >  \sum_{B\in {\mQ}} \frac{|B|}{n}\ .
\]
It follows that $\sum_{B\in {\mQ}} |B| < n$, as asserted in (\ref{bsizeTmuQ}).

Since any node in a finite tree precedes to at least one terminal node, we have
\begin{eqnarray*}
\# \thd  &=& 1+ \sum_{B\in \thd, |B|>0} 1
\le 1+ \sum_{B\in \thd, |B|>0} \#\{ B'\in {\mQ}: B\PQ B'\}
\\
&=&
1 +  \sum_{B'\in {\mQ}} \ \#\{B: |B|>0,B \PQ B'\} =
1 +  \sum_{B'\in {\mQ}} |B'|,
\end{eqnarray*}
thus (\ref{bsizeTmu}) follows from (\ref{bsizeTmuQ}).
\end{proof}

Till now the sequence $(\B_m)_{m\ge 0}$ of tree partitions could be quite general, i.e.~we
only assumed $\B_0=\{T\}$ and that $\B_m$ refines $\B_{m-1}$, $m\ge 1$. To proceed
we have to know something about the size of the sets $\B_m$. In particular, this is the case
if the $\B_m$ are constructed by root sets $R_m$ satisfying the properties of Proposition \ref{construct}.
Thus from now let us deal with those special tree partitions.

\begin{lem} \label{lb2} \ The number of subtrees of $\BB$
whose terminal set ${\mQ}$ satisfies $(\ref{bsizeTmuQ})$ does not exceed
$(8e)^n$.
\end{lem}

\begin{proof}
Since a subtree is entirely defined by its terminal set, we have to
find out how many sets ${\mQ}$ satisfy (\ref{bsizeTmuQ}). Denote
$q_m=\#\{{\mQ}\cap \B_m\}$. Then (\ref{bsizeTmuQ}) writes as
\be
\label{summ}
\sum_m m\, q_m <  n.
\ee
Since $q_m < \frac n m$, the number of non-negative integer
solutions of this inequality does not exceed
\be
\label{card1}
\prod_{m=1}^{n-1} (1+\frac nm) \le \prod_{m=1}^{n-1}   \frac {2n}m
=\frac{(2n)^n}{n!} \le \frac{(2n)^n}{(n/e)^n}= (2e)^n.
\ee
Recall the bound for the size levels (\ref{sizeRm}) which yields
\[
 \# \B_m =\# R^m \le 2^{m+1}.
\]
Thus, for a given sequence $q_m$, while constructing a terminal set ${\mQ}$, on
the $m$-th level $\B_m$ we have to choose $q_m$ elements from at most
$2^{m+1}$ elements of $\B_m$. Therefore, because of (\ref{summ}),  the number of possible sets
does not exceed
\be
\label{card2}
\prod_{m=1}^{n-1} \left( { {2^{m+1}} \atop  {q_m} } \right) \le
\prod_{m=1}^{n-1} (2^{m+1})^{q_m} = 2^{\sum_{m=1}^{n-1} (m+1) q_m}
\le 2^{2n}.
\ee
Combining (\ref{card1}) and (\ref{card2}) leads to the desired estimate.
\end{proof}

\subsection{Approximating Operators}
\label{sub44}

As mentioned at the beginning of Section \ref{proof} our objective is to find
families of operators from $\ell_1(T)$ into $\ell_q(T)$
approximating $W$ in a pointwise sense and where we are able to control their entropy numbers.
We are going to
construct those families now.

Let $(\B_m)_{m\ge 0}$ be a sequence of tree partitions with root sets $(R_m)_{m\ge 0}$ as in Proposition
\ref{construct}. Fix $n\ge 1$ and, given $\mu\in\ell_1(T)$ with $\norm{\mu}_1\le 1$, we define
the tree partition $\L$ and the subtree $\thd$ as in (\ref{extlight}) and (\ref{heavyall}), respectively.
Set
$$
\bbL_n:=\set{\L : \mu\in\ell_1(T)\,,\;\norm{\mu}_1\le 1}
$$
(recall that $n$ plays an important role in the construction of heavy and light domains,
thus the $\L$ really depend on $n$)
and observe that by Lemma \ref{lb2}
$$
\# \bbL_n \le (8\ex)^n\;.
$$
Fix $\LL\in \bbL_n$. We are going to define elements $(r_L^\circ)_{L\in\LL}$,
$(r_L^-)_{L\in \LL}$ and $(r_L^\bullet)_{L\in\LL}$ as follows:\\
Take $L\in\LL$ which may be represented as
$L=B_{r,m}$ for some $m\ge 1$ and $r\in R_m$. Then we set
\bee
\item
$r_L^\circ:=r$, i.e., $r_L^\circ$ is the root of the tree $L$.
\item
If $r_L^\circ\not=\ro$, then by $r_L^-$ we denote the parent element of $r_L^\circ$.
\item
Finally, let $B_{r',m-1}$ be the parent element (in $\BB$) of $L=B_{r,m}$. Then we put
$r_L^\bullet:=r'$.
\eee
Two cases may appear.

\noindent
\textit{Generic case}: A set $L\in \LL$ is called \textit{generic} provided that
$r_L^\bullet\pc r_L^\circ$. Note that then even $r_L^\bullet\pq r_L^-\pc r_L^\circ$.

\noindent
\textit{Degenerated case}: A set $L\in \LL$ is called
\textit{degenerated} if $r_L^\bullet= r_L^\circ$, i.e., if $L$ and its parent element (in $\BB$) possess the same root.

\medskip

Fix $\LL\in \bbL_n$. We define now four operators $W_\LL^1,\ldots,W_\LL^4$ depending on $\LL$ and acting
from $\ell_1(T)$ to $\ell_q(T)$ so that
\be
\label{Wasasum}
   W=\sum_{i=1}^4 W_\LL^i \ .
\ee
Given $s\in T$ we denote by $\delta_s\in\ell_1(T)$ the unit vector at $s$, i.e., $\delta_s(t)=0$
if $t\not=s$ and $\delta_s(s)=1$. Then the operator $W$ defined in (\ref{defW})
is completely described by
\be \label{defW2}
    (W\,\delta_s) (t) = \s(s) \al(t)  \, \on_{\{t\equiv s\}}\, \on_{[\ro,s]}(t)\  , \qquad s,t\in T.
\ee
The representation of $W$ as a sum is related to a splitting of the branch $[\ro,s]$ (and of the corresponding
indicator $\on_{[0,s]}$ which appears in (\ref{defW2})) in four pieces as described below.

For each $s\in T$ choose the unique element $L\in\LL$ such that $s\in L$.
If this light set $L$ with $s\in L$ is generic, then we split $[\ro,s]$ as follows:
\be \label{fourpieces}
[\ro,s]=[\ro,r_L^\bullet]
      \cup  (r_L^\bullet,r_L^-]
      \cup \set{r_L^\circ}
      \cup (r_L^\circ,s]\;.
\ee
Accordingly, we let \bea
\label{W1}
(W_\LL^1\,\delta_s) (t) &:=& \s(s) \al(t)  \, \on_{\{t\equiv s\}}\, \on_{[\ro,r_L^\bullet] }(t)\  , \qquad t\in T; \\
\label{W2}
(W_\LL^2\,\delta_s) (t) &:=& \s(s) \al(t)  \, \on_{\{t\equiv s\}}\, \on_{(r_L^\bullet,r_L^-]}(t)\  , \,\;\quad t\in T; \\
\label{W3}
(W_\LL^3\,\delta_s) (t) &:=& \s(s) \al(t)  \, \on_{\{t\equiv s\}}\, \on_{\set{r_L^\circ}}(t)\  , \qquad t\in T; \\
\nonumber
(W_\LL^4\,\delta_s) (t) &:=& \s(s) \al(t)  \, \on_{\{t\equiv s\}}\, \on_{(r_L^\circ,s]}(t)\  , \qquad t\in T.
\eea
But if $L\in \LL$ with $s\in L$ is degenerated, i.e., if $r_L^\bullet= r_L^\circ$,
then we simply have
$$
 [\ro,s]=[\ro,r_L^\circ] \cup (r_L^\circ,s]= [\ro,r_L^\bullet] \cup (r_L^\circ,s]\;.
$$
Accordingly, we define $W_\LL^1\delta_s$ and $W_\LL^4\delta_s$ as in the generic case, while now $W_\LL^2\delta_s=W_\LL^3\delta_s=0$.
The representation (\ref{Wasasum}) is straightforward.

 Setting
$$
W_\LL := \sum_{i=1}^3 W_\LL^i
$$
we get $W - W_\LL= W_\LL^4$, hence in view of $\log (\# \bbL_n) \le c\,n$
it suffices to prove that for $\mu\in\ell_1(T)$ with $\norm{\mu}_1\le 1$ always
\be
\label{estW4}
\inf_{\LL\in\bbL_n}\norm{W_\LL^4\, \mu}_q\le c\, n^{-\QQ}
\ee
as well as
\be
\label{estWLL}
\sup_{\LL\in \bbL_n}\,e_{[\rho n]}(W_\LL)\le c\,n^{-\QQ}
\ee
for a certain $\rho\ge 1$.

\medskip
We begin with proving the first assertion.

\begin{prop}
\label{pestW4}
There is a $c=c(q)$ such that for each $\mu\in\ell_1(T)$ with $\norm{\mu}_1\le 1$ inequality $(\ref{estW4})$
holds.
\end{prop}

\begin{proof}
In a first step we estimate $\norm{W_\LL^4 \mu}_q$ for \textit{arbitrary} $\LL\in\bbL_n$ and
$\mu\in\ell_1(T)$.
Here we have
\begin{eqnarray*}
\Big\|W_\LL^4 \mu \Big\|_q^q
&=& \Big\|\sum_{s\in T} \mu(s) W_\LL^4 \delta_s\Big\|_q^q
\\
&=&
\Big\|\sum_{L\in \LL} \sum_{s\in L}
\mu(s)W_\LL^4 \delta_s\Big\|_q^q\; .
\end{eqnarray*}

 Notice that the interior sums represent elements of $\ell_q(T)$ with disjoint supports
 (each sum is supported by the corresponding domain $L$). Hence, using the definitions
 of $W_\LL^4$ and of $d_\II$ it follows that
 \begin{eqnarray} \nonumber
\Big\|W_\LL^4 \mu\Big\|_q^q
&=&
\sum_{L\in \LL} \Big\|\sum_{s\in L}
\mu(s)W_\LL^4  \delta_s\Big\|_q^q
\le
\sum_{L\in  \LL} \left[ \sum_{s\in L}
|\mu(s)| \ \Big\|W_\LL^4 \delta_s\Big\|_q\right]^q
\\ \nonumber
&=&
\sum_{L\in  \LL} \left[ \sum_{s\in L} |\mu(s)| \ d_\II(r^\circ_L,s) \right]^q
\le
\sum_{L\in \LL} \left[ |\mu|\left(L\right)^q \sup_{s\in L}  d_\II(r^\circ_L,s)^q \right]
\\ \label{appro_oper}
&\le&
\sum_{L\in \LL}  \left[ |\mu|\left(L\right)^q   \eps_{|L|}^q \right]
\qquad
 \end{eqnarray}
where in the last step we used (\ref{net2}).

Estimate (\ref{appro_oper}) holds for any $\LL\in\bbL_n$ and $\mu\in\ell_1(T)$. Next, given
$\mu\in\ell_1(T)$ with $\norm{\mu}_1\le 1$,
we specify $\LL$ by taking $\LL=\L$ for the given $\mu$.
By the construction of $\L$ this yields
$ |\mu|(L) \le \frac{|L|}{n}$ for each $L\in \L$, cf.  (\ref{light}) and  (\ref{extlight}).
Then (\ref{appro_oper}) can be further estimated
as follows:
\begin{eqnarray*}
  \Big\|W_{\LL_{\mu}}^4 \mu\Big\|_q^q
  &\le&
  \sum_{L\in \L}  \left[ |\mu|\left(L\right)^q    \eps_{|L|}^q \right]
  \\
  &=&
  \sum_{L\in \L}  \left[ |\mu|\left(L\right) \cdot |\mu|\left(L\right)^{q-1}    \eps_{|L|}^q \right]
  \\
  &\le&
  \sum_{L\in \L} |\mu|\left(L\right) \cdot
  \sup_{L\in \L}  \left[ |\mu|\left(L\right)^{q-1}    \eps_{|L|}^q \right]
 \\
  &\le&
  \|\mu\|_1 \cdot
  \sup_{L\in \L}  \left[ \left(  \frac{|L|}{n} \right)^{q-1}   \left(\log 2\, |L|\right)^{-(q-1)} \right]
  \\
  &\le& c \ n^{-(q-1)}.
  \end{eqnarray*}
Thus our calculations result in
\[
    \Big\|W_{\LL_{\mu}}^4 \mu\Big\|_q \le   c \ n^{-\QQ}
\]
which, of course, implies (\ref{estW4}) and completes the proof.
\end{proof}
\bigskip

Our next objective is to verify (\ref{estWLL}). Recall that $W_\LL = W_\LL^1+ W_\LL^2+W_\LL^3$ with
$W_\LL^i$, $i=1,2,3$, defined in (\ref{W1}), (\ref{W2}) and (\ref{W3}), respectively. By the additivity of
the entropy number this implies
$$
e_{3n-2}(W_\LL)\le e_n(W_\LL^1)+ e_n(W_\LL^2)+e_n(W_\LL^3)\;.
$$
Thus, if we are able to verify $e_n(W_\LL^i)\le c_i\,n^{-\QQ}$ for  $i=1,2,3$, then this leads to
$$
e_{3n-2}(W_\LL)\le c\,n^{-\QQ}\,,
$$
hence (\ref{estWLL}) is valid with $\rho=3$.
Consequently, it suffices to estimate $e_n(W_\LL^i)$ for $i=1,2,3$
separately. We start with estimating
$e_n(W_\LL^1)$.

\begin{prop}
\label{pestW1}
There is a constant $c=c(q)$ such that
$$
   e_n(W_\LL^1)\le c\,n^{-\QQ}\;.
$$
\end{prop}

\begin{proof}
For $s\in T$ let $L$ be the unique domain in $\LL$ with $s\in L$. Clearly, then
$r_L^\bullet\pq s$, hence, if $s\not\equiv r_L^\bullet$, then it follows $W_\LL^1\delta_s=0$.
Thus it suffices to treat the case $s\equiv r_L^\bullet$ and then
\be
\label{repW1}
   (W_\LL^1\delta_s)(t)= \al(t)\sigma(r_L^\bullet)\  \on_{\{t\pq r_L^\bullet\,,\;t\equiv r_L^\bullet\}}  \;.
\ee
Let $\Upsilon_\LL^\bullet:=\set{r_L^\bullet : L\in\LL}$ and define an operator $V_\LL$
from $\ell_1(\Upsilon_\LL^\bullet)$ into $\ell_q(T)$ by
\be
\label{Up}
    (V_\LL \delta_{r_L^\bullet})(t):=\al(t)\sigma(r_L^\bullet)\ \on_{\{t\pq r_L^\bullet\,,\;t\equiv r_L^\bullet\}} \;.
\ee
Then, if $U_1$ is the unit ball in $\ell_1(T)$, by (\ref{repW1}) and
(\ref{Up}) it follows that $W_\LL^1(U_1)=V_\LL(U_1)$, hence $e_n(W_\LL^1)=e_n(V_\LL)$. In order to estimate
the latter entropy numbers we will use
the following convenient result from \cite[Proposition 1]{C}.
It provides a control of the entropy numbers for operators from $\ell_1$--spaces
into those of type $q$ based on the dimension of the first space. We refer to \cite{MaP} or \cite{Pis} for the definition of type $q$ .

\begin{prop}
Let $V$ be an operator from $\ell_1^N$ into a Banach space $X$ of type $q$. Then
for all $n=1,2,\ldots$ it follows that
\[
  e_n(V) \le c(X) f(n,N,q)\,\norm{V}
\]
where the constant $c(X)$ depends only on the type $q$--constant of the space $X$ and
$$
f(n,N,q):=
2^{-\max(n/N;1)}\min\set{1;\left[\max\left(\frac{\log\left(\frac N n+1\right)}{n};\frac 1 N\right)
\right]^{1-\,1/q}}\;.
$$
\end{prop}
Suppose now $n\ge N$. Then
$$
f(n,N,q)= 2^{-n/N}\,N^{-\QQ}\le c(q)\,n^{-\QQ}
$$
and we arrive at
\be
\label{Cm}
e_n(V)\le c(X,q)\,n^{-\QQ}
\ee
whenever $n\ge N$. Here $c(X,q)$ only depends on $q$ and the type $q$--constant of $X$.

In our case the operator $V:=V_\LL$ is defined on
$\ell_1^{N_\LL}=\ell_1(\Upsilon_\LL^\bullet)$ and acts into the space $\ell_q(T)$, which
(cf.~\cite{MaP}) for $1<q\le 2$ is
of type $q$ with type $q$--constant bounded by $\sqrt q$.
Therefore, the important dimension parameter is
$N_\LL:=\# \Upsilon_\LL^\bullet $.
Here we have
by (\ref{bsizeTmuQ})
\be
\label{dimY}
N_\LL=\# \Upsilon_\LL^\bullet \le \#\thd\le n
\ee
where $\mu\in\ell_1(T)$ and $\LL\in \bbL_n$ are related via $\LL=\L$.
Thus (\ref{Cm}) applies
and leads to
\[
   e_n(W_\LL^1) =  e_n(V_\LL) \le c\  ||V_\LL|| \ n^{-(1-1/q)} \le
   c\  ||W|| \ n^{-\QQ}
\]
with $c$ only depending on $q$.
In view of Proposition \ref{red2}, this completes the proof of Proposition \ref{pestW1}
by $\|W\|\le 2\,\|V_{\al,\s}\|\le 2\,\kappa$.
\end{proof}
\medskip

Our next objective is to estimate $e_n(W_\LL^2)$. Here Proposition \ref{crucial} will play an
important role.

\begin{prop}
\label{pestW2}
There is a constant $c=c(q)$  such that
\be
\label{estW2}
e_n(W_\LL^2)\le c\,n^{-\QQ}\;.
\ee
\end{prop}

\begin{proof}
Take $s\in T$ and choose as before the corresponding $L\in \LL$ with $s\in L$. In the case
that $L$ is degenerated we have $W_\LL^2 \delta_s=0$, thus it suffices to investigate those
$s\in T$ for which the corresponding $L$ is generic, i.e., we have $r_L^\bullet\pq r_L^-\pc r_L^\circ$.
Furthermore, whenever $s\not\equiv r_L^-$, then $W_\LL^2 \delta_s=0$ as well. On the other hand,
if $s\equiv r_L^-$, then we have $\s(s)=\s(r_L^-)=\s(r_L^\circ)$ and
\[
   \on_{\{ r^\bullet_L \prec t \preceq r^-_L\,,\; t\equiv s\}} =
   \on_{\{ r^\bullet_L \prec t \preceq r^-_L\,,\; t\equiv  r^\circ_L\}}\;.
\]
For generic $L\in\LL$ we define elements $x_L\in\ell_q(T)$ by
\[
  x_L(t)  :=  \alpha(t) \sigma(r^\circ_L) \on_{\{ r^\bullet_L \prec t \preceq r^-_L, t\equiv  r^\circ_L\}}
\]
and a set $C_\LL\subseteq \ell_q(T)$ by
$$
C_\LL:=\set{x_L : L\in \LL\;\mbox{is generic}}\;.
$$
Doing so gives
$$
e_n(W_\LL^2)= e_n(\mathrm{aco}(C_\LL))
$$
where $\mathrm{aco}(C_\LL)$ denotes the absolutely convex hull of $C_\LL\subseteq\ell_q(T)$.

Take a generic $L\in\LL$. Then there is an $m\ge 1$ such that $L=B_{\tau,m}$ with $\tau:=r_L^\circ\in R_m$.
Setting $r:=r_L^\bullet$, then $r\in R_{m-1}$,
$r\pc \tau$,
$\tau\in R_m\setminus R_{m-1}$ and $\tau\in B_{r,m-1}$. Thus we are exactly in the
situation of Proposition \ref{crucial} with $\tau=r_L^\circ$ and $r=r_L^\bullet$ which implies
\be
\label{estxL}
\norm{x_L}_q =\s(r_L^\circ)\Big(\sum_{r_L^\bullet\pc v\pq r_L^-\atop{v\equiv r_L^\circ}}\al(v)^q\Big)^{1/q}
\le c\,m^{-1}=c\,|L|^{-1}\;.
\ee
Hence, for any $h>0$
\be
\label{estcard}
  \#\left\{ L\in \LL : ||x_L||_q\ge h \right\} \le \#\left\{ L\in \LL: |L|\le \frac c h \right\} \le 2^{c/h +2}
\ee
where we used $\#\{L\in \LL : |L|=m\}\le \# R_m \le 2^{m+1}$ in the last estimate.

By \cite[Proposition 6.2]{CKP}, which handles the entropy of convex hulls in type $q$ spaces in the non-critical case,
estimate (\ref{estcard}) yields
\[
   e_k \left(\mathrm{aco}(C_\LL) \right) \le c\ k^{-\QQ} (\log k)^{-1/q}, \qquad k\ge 1\;.
\]
For $k=n$ we have
\[
   e_n(W_\LL^2) = e_n \left(\mathrm{aco}(C_\LL) \right) \le c\ n^{-\QQ} (\log n)^{-1/q}\;.
\]
This is even slightly better than required in (\ref{estW2}) and completes the proof.
\end{proof}
\medskip

Our final objective is to estimate $e_n(W_\LL^3)$ suitably.

\begin{prop}
\label{pestW3}
There is a constant $c=c(q)$ such that
$$
   e_n(W_\LL^3)\le c\,n^{-\QQ}\;.
$$
\end{prop}

\begin{proof}
Take $s\in T$ and $L\in\LL$ with $s\in L$. If $L$ is degenerated, then $W_\LL^3\delta_s=0$.
This is so too if $s\not\equiv r_L^\circ$. Consequently,
\beaa
\set{W_\LL^3 \delta_s : s\in T}&=&
\set{\s(s)\al(r_L^\circ)\delta_{r_L^\circ} : s\equiv r_L^\circ\;,L\,\mbox{with}\;s\in L\;\mbox{is generic}\;,s\in T}\cup\set{0}\\
&=&\set{\s(r_L^\circ)\al(r_L^\circ)\delta_{r_L^\circ} : L\;\mbox{is generic}}\cup\set{0}\;.
\eeaa
Set
$$
G_\LL:=\set{r_L^\circ : L\;\mbox{is generic}}
$$
and define a diagonal operator $D_\LL^3: \ell_1(G_\LL)\mapsto \ell_q(G_\LL)$ by
$$
D_\LL^3(\delta_{r_L^\circ}):=\gamma_L \,\delta_{r_L^\circ}\;,\quad L\;\mbox{generic}\;,
$$
where $\gamma_L:= \s(r_L^\circ)\al(r_L^\circ)$ . Then $e_n(W_L^3)=e_n(D_\LL^3)$ and it suffices to
estimate the $\gamma_L$ suitably.

Recall that $r^\circ_L$ and $r^\bullet_L$ belong to the same element of the partition
$\B_{|L|-1}$, hence, if  $L$ is generic, i.e., if $r^\bullet_L \pc r^\circ_L$,  by (\ref{net2})
 we obtain
$$
\gamma_L:= \s(r_L^\circ)\al(r_L^\circ)\le d_\II(r_L^\bullet,r_L^\circ)\le \eps_{|L|-1}\;.
$$
It follows that
\[
  \#\left\{ L: \dg_L\ge \eps_m \right\} \le \#\left\{ L: |L|\le m+1 \right\} \le 2^{m+3}.
\]
Again we used $\#\{L\in\LL : |L|=m\}\le \# R_m \le 2^{m+1}$ in the last step.
If $\{\dg^*_k\}_{k\ge 1}$ is the non-increasing rearrangement of
$\{\dg_L\}_{L\in \LL}$, we have
\[
    \dg^*_k \le c\ (\log k)^{-\QQ}.
\]
By using \cite[Proposition 3.1]{Kuhn} where the entropy of critical
diagonal operators with logarithmic diagonal is handled, we obtain
\[
  e_k(W_\LL^3)=e_k(D_\LL^3)\le c\ k^{-\QQ},\qquad  k\ge 1.
\]
For $k=n$ we have
\[
  e_n(W_\LL^3)=e_n(D_\LL^3)\le c\ n^{-\QQ}
\]
as asserted.
\end{proof}
\medskip

\section{Final Remarks}
We must acknowledge that the proof of Theorem \ref{main} or Theorem \ref{main1}, respectively,
is quite complicated. One of the reasons for this is that
so many operators $W_\LL^i$ are involved. Thus a natural question is why do not two or three operators suffice.
Indeed, once a very natural bound for $W_\LL^4$ (Proposition \ref{pestW4}) is obtained,
it is tempting to use (for the generic case) a splitting into two pieces instead into four
 as in (\ref{fourpieces}), i.e., to split $[\ro,s]$ only as
\[
  [\ro,s]=[\ro,r_L^\circ] \cup (r_L^\circ,r].
\]
In different words, why we cannot add up $W_\LL^1,W_\LL^2,W_\LL^3$ into one operator and to
deal with it as we did with $W_\LL^1$ ? In fact, the corresponding bound is dimension--based.
Therefore, proceeding in this way, we must replace the dimension bound (\ref{dimY})
with some bound for $\# \Upsilon_\LL^\circ $ where
$\Upsilon_\LL^\circ:=\{ r_L^\circ :\ L\in \LL\}$. Unfortunately,
$\# \Upsilon_\LL^\circ=\#\LL$, the number of \textit{extremal light} domains,
does not admit any uniform estimate, unlike the number of \textit{heavy} domains we used in the proof.
The only chance to estimate $\# \Upsilon_\LL^\circ$ is to make further assumptions about the structure
of the underlying tree $T$.
Therefore, the splitting into two pieces does not work for general trees.

Once this difficulty is understood, the next natural idea is to use a splitting into three pieces,
\[
  [\ro,s]=[\ro,r_L^\bullet] \cup (r_L^\bullet, r_L^\circ] \cup (r_L^\circ,r].
\]
In different words, why we cannot add up $W_\LL^2$ and $W_\LL^3$ into one operator and to
deal with it as we did with $W_\LL^2$ ? Recall that the corresponding bound from
Proposition \ref{pestW2} is based on the size evaluation $||x_L||_q\le c\,|L|^{-1}$
from (\ref{estxL}). That one in turn was built upon the tricky property (4) from Proposition
\ref{construct}. Once we use the three pieces splitting, we can only use (\ref{net2})
for the evaluation of $||x_L||$, as we did when working with $W_\LL^3$. On this way
we only obtain $||x_L||\le \eps_{|L|-1}= c\,|L|^{-\QQ}$. Unfortunately, we \textit{do not know} whether
or not this weaker bound provides the necessary bound $c\, n^{-\QQ}$ for the entropy numbers
of the convex hull of a sequence. To the best of our knowledge, the required result
is missing in the literature for subsets of spaces of type $q$ (or even for subsets of some
$\ell_q$--spaces). Thus it is this knowledge gap that forced us to struggle
with partition constructions possessing the mentioned property (4), and then extract the
well studied diagonal operators $W_\LL^3$ or $D_\LL^3$ by the further splitting
$(r_L^\bullet, r_L^\circ] = (r_L^\bullet, r_L^-] \cup \{r_L^\circ\}$, eventually
coming to the proof presented here. Let us still mention that for $q=2$, i.e., for
sets in Hilbert spaces, such entropy estimates for convex hulls of sequences are known (cf.~Proposition 4 in \cite{CE}).
Hence, if $q=2$, we may add up $W_\LL^2$ and $W_\LL^3$ into one operator which slightly simplifies
the proof in that case because here we neither need property (4) of Proposition \ref{construct} nor Proposition \ref{crucial}.

Another difficulty comes from the partition $\II$ of $T$ generated by the weight $\s$. This forced us
to replace the metric $d$ on $T$ by the localized "distance" $d_\II$. Of course, this additional difficulty does
not appear for one--weight summation operators $V_{\al,\s}$ with $\s(t)=1$, $t\in T$. Hence, also in
that case the proof of Theorem \ref{main} becomes slightly less involved.
\bigskip

\noindent
\textbf{Acknowledgement:}
\ The research was supported by the RFBR-DFG grant 09-01-91331 "Geometry and asymptotics of random structures".
The work of the first named author was also supported by RFBR grants 09-01-12180-ofi\_m and
10-01-00154a, as well as by Federal Focused Programme 2010-1.1-111-128-033.

We are grateful to M.Lacey for his interest and useful discussions that led us to the present research.

\bibliographystyle{amsplain}

\vspace{1cm}

\parbox[t]{7cm}
{Mikhail Lifshits\\
St.Petersburg State University\\
Dept Math. Mech. \\
198504 Stary Peterhof, \\
Bibliotechnaya pl., 2\\
Russia\\
email: lifts@mail.rcom.ru}\hfill
\parbox[t]{6cm}
{Werner Linde\\
Friedrich--Schiller--Universit\"at Jena \\
Institut f\"ur Stochastik\\
Ernst--Abbe--Platz 2\\
07743 Jena\\
Germany\\
email: werner.linde@uni-jena.de}

\end{document}